\documentclass[a4paper]{amsart} 
\usepackage[english]{babel}
\usepackage[utf8x]{inputenc}
\usepackage[T1]{fontenc}

\usepackage[a4paper,top=3cm,bottom=2cm,left=3cm,right=3cm,marginparwidth=1.75cm]{geometry}

\usepackage{amsmath, amsthm}
\usepackage{amsfonts}
\usepackage{amssymb}
\usepackage{enumerate}
\usepackage{graphicx}
\usepackage{bm}
\usepackage[colorinlistoftodos]{todonotes}
\usepackage[colorlinks=true, allcolors=blue]{hyperref}
\usepackage{caption}

\newtheorem{theorem}{Theorem}
\theoremstyle{plain}
\newtheorem{proposition}[theorem]{Proposition}
\newtheorem{remark}[theorem]{Remark}
\newtheorem{lemma}[theorem]{Lemma}

\newtheorem{corollary}[theorem]{Corollary}
\numberwithin{theorem}{section}

\newcommand{\Q}{\mathbb{Q}}
\newcommand{\Z}{\mathbb{Z}}
\newcommand{\N}{\mathbb{N}}
\newcommand{\C}{\mathbb{C}}
\newcommand{\OK}[1]{\ensuremath{\mathcal{O}_{#1}}} 
\newcommand{\OKPlus}[1]{\ensuremath{\mathcal{O}_{#1}^+}}
\newcommand{\UK}[1]{\ensuremath{\mathcal{U}_{#1}}} 
 
\newcommand{\BQ}[2]{\ensuremath{\Q\big(\sqrt{#1}, \sqrt{#2}\big)}}

\newcommand{\Tr}[1]{\mathrm{Tr}\left(#1\right)} % trace
\newcommand{\norm}[1]{\mathcal{N}\left(#1\right)}

\newcommand{\A}{\alpha}
\newcommand{\B}{\beta}
\newcommand{\G}{\gamma}

\title{Universal Quadratic Forms and Indecomposables over Biquadratic Fields}
\author[M. \v{C}ech, D. Lachman, J. Svoboda, M. Tinkov\'{a}, K. Zemkov\'{a}]{Martin \v{C}ech, Dominik Lachman, Josef Svoboda, Magdal\'{e}na Tinkov\'{a}, Krist\'{y}na Zemkov\'{a}}
\address{Charles University, Faculty of Mathematics and Physics, Department of Algebra,
Sokolovsk\'{a} 83, 18600 Praha 8, Czech Republic}
\email{martinxcech@gmail.com, dominiklachman@seznam.cz, josefsvobod@gmail.com, tinkova.magdalena@gmail.com, zemk.kr@gmail.com}

\date{\today}

\thanks{M. T. was supported by Czech Science Foundation (GA\v{C}R), grant 17-04703Y. M. \v{C}. and K. Z. were supported by student faculty grant. The authors were supported by the project SVV-2017-260456.   }

\begin{document}

\begin{abstract}
The aim of this article is to study (additively) indecomposable algebraic integers $\mathcal O_K$ of biquadratic number fields $K$ and universal totally positive quadratic forms with coefficients in $\mathcal O_K$.
There are given sufficient conditions for an indecomposable element of a quadratic subfield to remain indecomposable in the biquadratic number field $K$. Furthermore, estimates are proven which enable algorithmization of the method of escalation over $K$. These are used to prove, over two particular biquadratic number fields $\BQ 23$ and $\BQ{6}{19}$, a lower bound on the number of variables of a universal quadratic forms, verifying Kitaoka's conjecture.
\end{abstract}

\maketitle

\section{Introduction}

Quadratic forms have been a subject of interest to many researchers. For example a famous theorem of Lagrange states that the quadratic form $x_1^2+x_2^2+x_3^2+x_4^2$ is universal over $\Z$, i.e., that every positive integer can be expressed as a sum of four squares. Moreover, it is known that $\Z$ does not admit a ternary universal quadratic form.

The study of universal quadratic forms has been generalized to other fields than $\Q$. Siegel \cite{Si} showed that the only totally real number fields in which the sum of any number of squares is universal are $\Q$ and $\Q(\sqrt{5})$, three squares being sufficient in the latter case. Chan, Kim and Raghavan \cite{CKR} characterized ternary universal forms over $\Q(\sqrt{2})$, $\Q(\sqrt{3})$ and $\Q(\sqrt{5})$ up to equivalence and showed, that no other totally real quadratic number field admits a totally positive ternary universal form. More results can be found in \cite{De}, \cite{Sa}, \cite{Ki1} or \cite{Ki2}. A recent result of Blomer and Kala \cite{BK} shows that, for any positive integer $N$, there are infinitely many totally real quadratic number fields which do not admit any universal quadratic form with less than $N$ variables. This result was generalized by Kala and Svoboda \cite{KS} to multiquadratic fields of any fixed degree.

Closely related to universal quadratic forms are indecomposable elements, i.e., totally positive integers which cannot be written as a sum of two other totally positive integers. These are hard to be represented by a quadratic form so they often must appear as its coefficients, thus increasing the required number of variables. This idea played a crucial role in the proof of Blomer and Kala. The set of indecomposables was fully characterized over totally real quadratic fields $\Q(\sqrt{p})$ using properties of the continued fraction of $\sqrt{p}$ (see \cite{DS} or \cite{Pe}). The structure of the additive semigroup of totally positive elements in quadratic number fields was described in \cite{HK}.

We will focus on quadratic forms over totally real biquadratic fields, i.e., quadratic forms with coefficients in the ring of integers of a field of the form $\BQ{p}{q}$ with squarefree positive integers $p,q$. We will show that, under certain conditions, the indecomposable elements from the quadratic subfields of $\BQ p q$ remain indecomposable in the biquadratic field. We are going to prove the following theorem (note that the case distinction used in (b) depends on the values of $p,q\pmod 4$ and is described in the next section).

\begin{theorem}\label{Indecomposables}
Let $K=\BQ p q $ be a biquadratic number field, and let $r=\frac{pq}{\gcd(p,q)^2}$. For $k\in\{p,q,r\}$, set $M_k=\max\{u_i;\; i\text{ odd, }[u_0,\overline{u_1,u_2,\dots, u_{s-1}, u_s}] \text{ is the continued fraction of } -\overline{\omega_k} \}$.
\begin{enumerate}[(a)]
\item Let $\alpha\in\Q(\sqrt{p})$ be indecomposable. 
	\begin{itemize}
		\item If $\alpha$ is a convergent of $-\overline{\omega_p}$ and $\sqrt{r}>\sqrt{p}$, then $\alpha$ is indecomposable in $K$.
        \item If $\sqrt{r}>M_p\sqrt{p}$, then $\alpha$ is indecomposable in $K$.
    \end{itemize}

\item Let $\beta\in\Q(\sqrt{q})$ be indecomposable.
	\begin{itemize}
    	\item In the cases 1., 2., 3., $\beta$ is indecomposable in $K$.
        \item In the case 4., if $\sqrt{r}>\sqrt{q}$ and $\beta$ is a convergent of $-\overline{\omega_q}$, then $\beta$ is indecomposable in $K$.
        \item In the case 4., if $\sqrt{r}>M_q\sqrt{q}$, then $\beta$ is indecomposable in $K$.
    \end{itemize}
\item  Let $\gamma\in\Q(\sqrt{r})$ be indecomposable.
	\begin{itemize}
    	\item If $\gamma$ is a convergent of $-\overline{\omega_r}$ and $\sqrt{p}>\sqrt{r}$, then $\gamma$ is indecomposable in $K$.
        \item If $\sqrt{p}>M_r\sqrt{r}$, then $\gamma$ is indecomposable in $K$.
    \end{itemize}
\end{enumerate}
\end{theorem}

Since we do not have a characterization of indecomposable integers in biquadratic fields, this theorem is the first step in understanding of their structure.

In Section 2 we have compiled some basic facts about biquadratic fields, indecomposable integers and quadratic forms over totally real fields. Section 3 presents some original results about indecomposable integers in biquadratic fields. Moreover, it provides proof that the totally positive integer of sufficiently small norms are indecomposable in these fields. We also prove the previous theorem and decide in which biquadratic fields we can eventually decompose the indecomposable integers from quadratic subfields.    

Sections 4 and 5 are devoted to using this theorem and the escalation method to find a lower bound on the number of variables of a universal quadratic form over a biquadratic field. Section 4 is preparatory, where we make necessary estimates which allow us to solve the problem algorithmically. In Section 5, two particular number field are treated, and we show that any classical universal totally positive quadratic form over $\BQ23$ ($\BQ{6}{19}$, resp.) must have at least 5 variables (6 variables, resp.); in the case of the field $\BQ23$, a detailed computation is provided. Note that these results confirm Kitaoka's conjecture that states that there are very few number fields that admit ternary universal forms.
%-----------------------------------------------------
\section{Preliminaries}

For a number field $F$, $\OK{F}$ denotes its ring of integers and $\OKPlus{F}$ is the semigroup of \emph{totally positive integers}, i.e., the elements $\alpha\in\OK{F}$ satisfying $\sigma(\alpha)>0$ for any embedding $\sigma$ of $F$ into $\C$. If $\alpha\in F$ is totally positive, we write $\alpha\succ0$, and similarly if $\beta\in F,$ we write $\alpha\succ\beta$ if $\alpha-\beta\succ0.$ The \emph{norm} and \emph{trace} of $\alpha\in F$ are defined respectively by $$\begin{aligned}\norm{\alpha}&=\prod_{\sigma}\sigma(\alpha),\\\Tr{\alpha}&=\sum_{\sigma}\sigma(\alpha),\end{aligned}$$ where $\sigma$ runs over all embeddings of $F$ into $\C$.

Throughout the article, $p$ and $q$ will denote squarefree positive integers, and the corresponding biquadratic field $\BQ{p}{q}$ will be denoted by $K$. Furthermore, fix $r=\frac{pq}{\gcd(p,q)^2}.$ Then $K$ has three quadratic subfields -- $\Q(\sqrt{p})$, $\Q(\sqrt{q})$ and $\Q(\sqrt{r})$. There are four embeddings of $K$ into $\C$: if $\alpha\in K$ is of the form $a+b\sqrt{p}+c\sqrt{q}+d\sqrt{r}$ with $a,b,c,d\in\Q$, the embeddings are the following:$$\begin{aligned}
\sigma_1(\alpha)&=a+b\sqrt{p}+c\sqrt{q}+d\sqrt{r}, \\
\sigma_2(\alpha)&=a-b\sqrt{p}+c\sqrt{q}-d\sqrt{r}, \\
\sigma_3(\alpha)&=a+b\sqrt{p}-c\sqrt{q}-d\sqrt{r}, \\
\sigma_4(\alpha)&=a-b\sqrt{p}-c\sqrt{q}+d\sqrt{r}.
\end{aligned}$$
 
Depending on $p,q\pmod{4}$, after possibly interchanging the role of $p$, $q$ and $r$, every case can be converted into one of the following (see \cite[Section 8]{Ja}):

\begin{enumerate}
\item $p\equiv 2\;(\text{mod } 4)$, $q\equiv 3\;(\text{mod } 4)$,
\item $p\equiv 2\;(\text{mod } 4)$, $q\equiv 1\;(\text{mod } 4)$,
\item $p\equiv 3\;(\text{mod } 4)$, $q\equiv 1\;(\text{mod } 4)$,
\item $p\equiv 1\;(\text{mod } 4)$, $q\equiv 1\;(\text{mod } 4)$ and
\begin{enumerate}
\item $\frac{p}{\gcd(p,q)}\equiv \frac{q}{\gcd(p,q)}\equiv1\;(\text{mod } 4)$ or
\item $\frac{p}{\gcd(p,q)}\equiv \frac{q}{\gcd(p,q)}\equiv3\;(\text{mod } 4)$.
\end{enumerate}
\end{enumerate}

We will therefore restrict ourselves to these cases whenever any specific property of the basis will be needed. An integral basis of $\OK{K}$ in the different cases has the following form (see \cite[Theorem 2]{Wi}):

\begin{enumerate}
\item $\left(1, \sqrt{p}, \sqrt{q}, \frac{\sqrt{p}+\sqrt{r}}2\right)$,
\item $\left(1, \sqrt{p}, \frac{1+\sqrt{q}}2, \frac{\sqrt{p}+\sqrt{r}}{2} \right)$,
\item $\left(1, \sqrt{p}, \frac{1+\sqrt{q}}2, \frac{\sqrt{p}+\sqrt{r}}{2}\right)$,
\item 
\begin{enumerate}
\item $\left(1, \frac{1+\sqrt{p}}2, \frac{1+\sqrt{q}}2, \frac{1+\sqrt{p}+\sqrt{q}+\sqrt{r}}{4}\right)$,
\item $\left(1, \frac{1+\sqrt{p}}2, \frac{1+\sqrt{q}}2, \frac{1-\sqrt{p}+\sqrt{q}+\sqrt{r}}{4}\right)$.
\end{enumerate}
\end{enumerate}

Note that in all cases, we have $p\equiv r\pmod{4}$ so in cases 1, 2 and 3, $p$ and $r$ are interchangeable.

An element $\alpha\in\OKPlus{K}$ is called \emph{indecomposable} if it cannot be written as $\alpha=\beta+\gamma$ with $\beta,\gamma\in\OKPlus{K}$. If $F=\Q\left(\sqrt{p}\right)$ is a quadratic field, it is possible to characterize the indecomposable elements in terms of the continued fraction of certain algebraic integer in $F$. Let us denote
\[
    \omega_p=\left\{
                \begin{array}{ll}
                  \sqrt{p} \qquad\text{if } p\equiv 2,3\;(\text{mod } 4),\\
                  \frac{1+\sqrt{p}}{2}\quad\text{if } p\equiv 1\;(\text{mod } 4).
                \end{array}
              \right.
\]
The set $\{1,\omega_p\}$ forms an integral basis of $\OK{F}$. We know that $-\overline{\omega_p}$, where $\overline{\omega_p}$ is the Galois conjugate of $\omega_p$, has an eventually periodic continued fraction of the form $[u_0,\overline{u_1,u_2,\dots, u_{s-1}, u_s}]$. By
\[
\frac{x_i}{y_i}=[u_0,\ldots,u_i],
\]
we mean the $i$th convergent of $-\overline{\omega_p}$. We will also use the expression \emph{convergent} to describe the algebraic integer $\alpha_i=x_i+y_i\omega_p$. Moreover, let us denote by $\alpha_{i,l}$ the elements $\alpha_{i,l}=\alpha_i+l\alpha_{i+1}$ with $0\leq l\leq u_{i+2}$, which we call the \emph{semiconvergents} of $-\overline{\omega_p}$. In \cite{Pe, DS}, it is proved that the indecomposable integers in $\Q(\sqrt{p})$ are exactly the totally positive semiconvergents of $-\overline{\omega_p}$, i.e., $\alpha_{i,l}$ with $i$ odd, and their conjugates.

In this article, we shall use the fact that the elements $\alpha_i$ are the best approximations of the second kind, see \cite{Kh}. It means that the convergents satisfy the inequality
\[
|x_i+y_i\overline{\omega_p}|<|u+v\overline{\omega_p}| 
\]
for all $u,v\in\N_0$ such that $1\leq v\leq y_i$ and $\frac{u}{v}\neq \frac{x_i}{y_i}$. Moreover, the only elements with this property are the convergents.

By a quadratic form over a totally real field $F$ with variables $x_1,\dots x_n,$ we mean an expression of the form $$Q(x_1,\dots,x_n)=\sum_{1\leq i\leq j\leq n}a_{ij}x_ix_j,$$ where $a_{ij}\in\OKPlus{F}.$ The quadratic form $Q(x_1,\dots,x_n)$ is called \begin{itemize}
\item \emph{totally positive definite} if $Q(\gamma_1,\dots ,\gamma_n)\succ0$ whenever $\gamma_1,\dots,\gamma_n\in\OK{F}$, not all zero,
\item \emph{diagonal} if $a_{ij}=0$ whenever $i\neq j$,
\item \emph{classical} if $a_{ij}$ is divisible by 2 whenever $i\neq j$, 
\item \emph{universal} if it represents all elements from $\OKPlus{F}$, i.e., for every $\alpha\in\OKPlus{F},$ there exist $\gamma_1,\dots \gamma_n\in\OK{F}$ such that $\alpha=\sum_{i,j=1}^na_{ij}\gamma_i\gamma_j.$
\end{itemize}

Every quadratic form $Q(x_1,\dots,x_n)$ with $n$ variables can be associated with a symmetric $n\times n$ matrix $$Q(x_1,\dots,x_n)=\begin{pmatrix}x_1 & \cdots & x_n\end{pmatrix}\cdot\begin{pmatrix}
a_{11} & \frac {a_{12}}2 & \cdots & \frac{a_{1n}}2\\
\frac{a_{12}}{2} & a_{22} & \cdots & \frac{a_{2n}}2\\
\vdots & \vdots & \ddots & \vdots \\
\frac{a_{1n}}2 & \frac{a_{2n}}2 & \cdots & a_{nn}
\end{pmatrix}\cdot\begin{pmatrix}x_1 \\ x_2 \\ \vdots \\ x_n\end{pmatrix}.$$ 

We see that the quadratic form is diagonal if and only if the associated matrix is diagonal, and it is classical if and only if all entries of the associated matrix are in $\OK{F}$. We shall often identify a quadratic form with its associated matrix.

Throughout the article, we will consider only totally positive definite classical quadratic forms.

%----------------------------------------------------
\section{Indecomposable elements over biquadratic fields}

As we have said, the indecomposable elements in quadratic fields can be fully described by a continued fraction, and they are fairly well-studied, see e.g. \cite{DS, JK, Ka2}. On the other hand, we are not aware of any such characterization in biquadratic fields (or in any other family of number fields). Hence, to understand the indecomposable elements in biquadratic fields, the obvious starting point is to look at the indecomposables in the three quadratic subfields. The question is if these elements remain indecomposable in the bigger biquadratic field. In this section, we positively answer the question if some additional conditions are satisfied, but in some cases, the question remains open.

Recall that $K=\BQ p q$, where $p, q$ are squarefree positive integers, and that $r=\frac{pq}{\gcd(p,q)^2}$. The following lemma gives us some useful inequalities for the coefficients of totally positive elements.

\begin{lemma} \label{lemma:totpos}
Let $\alpha=a+b\sqrt{p}+c\sqrt{q}+d\sqrt{r}\in\Q(\sqrt{p},\sqrt{q})$. If $\alpha$ is totally positive, then $a>|b|\sqrt{p}$, $a>|c|\sqrt{q}$ and $a>|d|\sqrt{r}$. 
\end{lemma}

\begin{proof}
We give the proof only for one of our inequalities, the other cases are similar. Since $\alpha$ is totally positive, we have $\sigma_i(\alpha)>0$ for all $i\in\{1,\ldots,4\}$. Thus
\[
0<\sigma_1(\alpha)+\sigma_3(\alpha)=2a+2b\sqrt{p}
\]
and
\[
0<\sigma_2(\alpha)+\sigma_4(\alpha)=2a-2b\sqrt{p}.
\]
Combining these two inequalities we obtain the desired conclusion. 
\end{proof} 

We proceed with a proof that elements with sufficiently small norm are indecomposable.

\begin{lemma}\label{stopa} 
Suppose that $\A = a+b\sqrt{p}+c\sqrt{q}+d\sqrt{r} \in \mathcal{O}_K^+$ (with $a, b, c, d\in\Q$). %Let $\Tr{}$ denote the trace from $K$ to $\Q$.
\begin{enumerate}[(a)]
\item If $b \neq 0$, then
$\Tr{\A} > \sqrt{p}$.
\item If $c \neq 0$, then
$\Tr{\A} > \sqrt{q}$.
\item If $d \neq 0$, then
$\Tr{\A} > \sqrt{r}$.
\end{enumerate}
Finally, if $\A \notin \mathbb{Z}$, then $\Tr{\A}> \min(\sqrt{p},\sqrt{q},\sqrt{r}).$
\end{lemma}

\begin{proof}
All the cases are analogous, so let us prove only the one when $b \neq 0$. Since $\A \in \mathcal{O}_K$, the coefficients $a, b, c, d$ are quarter-integers. The~element $\A$ is totally positive, so from Lemma \ref{lemma:totpos} we have that $a>|b|\sqrt{p}$. Since $b$ is a~nonzero quarter-integer, the~inequality $a>|b|\sqrt{p}$ implies that $a>\frac{\sqrt{p}}{4}$, and hence $\Tr{\A}=4a >\sqrt{p}.$
\end{proof}

\begin{proposition}	\label{normab}
Assume that $\A \in \mathcal{O}_K^+$ satisfies $\norm{\A}<2\min(\sqrt{p},\sqrt{q},\sqrt{r})$ and  $n \nmid \A$ for~any $n \in \mathbb{N}, n>1$. Then $\A$ is indecomposable.  
\end{proposition}

\begin{proof}
Denote by $\delta:=\min(\sqrt{p},\sqrt{q},\sqrt{r})$. For~contradiction, suppose that $\A=\B+\G,$ where $\B$, $\G \in \mathcal{O}_K^+$. Then we have:

\begin{align}
2\delta &> \norm{\A} = \big(\sigma_1(\B)+\sigma_1(\G)\big)\big(\sigma_2(\B)+\sigma_2(\G)\big)\big(\sigma_3(\B)+\sigma_3(\G)\big)\big(\sigma_4(\B)+\sigma_4(\G)\big)> \nonumber \\
&>
\Tr{\sigma_1(\B)\sigma_2(\G)\sigma_3(\G)\sigma_4(\G)}+
\Tr{\sigma_1(\B)\sigma_2(\B)\sigma_3(\B)\sigma_4(\G)} + \text{other (totally positive) summands}. \nonumber
\end{align}

\noindent The final part of Lemma \ref{stopa} implies that $\sigma_1(\B)\sigma_2(\G)\sigma_3(\G)\sigma_4(\G)$ or $\sigma_1(\B)\sigma_2(\B)\sigma_3(\B)\sigma_4(\G)$ must be a (positive) integer. 
Without loss of generality, suppose that it is the former. This element then equals each of its conjugates and we have
$$\sigma_1(\B)\sigma_2(\G)\sigma_3(\G)\sigma_4(\G)=\sigma_2(\B)\sigma_1(\G)\sigma_4(\G)\sigma_3(\G) \in \mathbb{N}.$$

\noindent If we divide the~equality by~the~norm of~$\G$ and multiply it by $\sigma_1(\G)$, we get
$$\B=\sigma_1(\B)=\frac{\sigma_2(\B)\sigma_1(\G)\sigma_4(\G)\sigma_3(\G)}{\norm{\G}} \sigma_1(\G) =\frac{u}{v} \G,$$
where $u$ and $v$ are coprime natural numbers. Then we have $\B=u\mu$ and $\G = v\mu$ for $\mu:=\G/v \in \mathcal{O}_K$, so $\A = (u+v)\mu$
is divisible by the integer $u+v\geq 2$, which gives us a contradiction.
\end{proof}

This proposition can be directly applied to obtain the indecomposability of units (which, in fact, holds in every totally real number field).

\begin{corollary}\label{Prop:indecompunits}
Every totally positive unit of $\BQ p q$ is indecomposable.
\end{corollary}

We show that every element, which is indecomposable in a quadratic field, remains indecomposable in all but finitely many biquadratic fields.

\begin{theorem}\label{Th:IndFinMany}
Let $x+y\sqrt{p}$ be an indecomposable element of $\OKPlus{\Q(\sqrt{p})}$. Then $x+y\sqrt{p}$ can decompose only in biquadratic fields $\BQ{p}{q}$ with $\min\{q, r\}<16x^2$.
\end{theorem}

\begin{proof}
Assume that $K=\BQ{p}{q}$ is such that $x+y\sqrt{p}$ is decomposable in $\OKPlus{K}$. Then we have
$$
x+y\sqrt{p} =\underbrace{ \left(a+b\sqrt{p}+c\sqrt{q}+d\sqrt{r} \right)}_{\alpha}+\underbrace{\left(a'+b'\sqrt{p}+c'\sqrt{q}+d'\sqrt{r} \right)}_{\alpha'}
$$
for some $a,b,c,d, a', b', c', d'\in\Q$ such that $\alpha, \alpha'\in\OKPlus{K}$.  Since $\alpha$ and $\alpha'$ are totally positive, we have $a,a'>0$. As $a+a'=x$, it must hold that $a,a'<x$. 
Note that $\frac{x}{2}+\frac{y}{2}\sqrt{p}\notin\OK{K}$ whenever $\frac{x}{2}+\frac{y}{2}\sqrt{p}\notin\OK{\Q(\sqrt{p})}$, and hence $c$ or $d$ must be nonzero. We can use Lemma \ref{stopa}: If $c\neq0$, then $\sqrt{q}<\Tr{\alpha}=4a<4x$; if $d\neq0$, then $\sqrt{r}<\Tr{\alpha}=4a<4x$. Hence, $q<16x^2$ or $r<16x^2$.
\end{proof}

\begin{remark}
In the proof of Lemma \ref{stopa}, we used that $a,b,c,d$ are quarter-integers. In the cases 1., 2., 3. (i.e., if $p\equiv2,3\pmod4$), $a,b,c,d$ are half-integers; therefore, we can get in Theorem \ref{Th:IndFinMany} in these cases a better bound $\min\{q,r\}<4x^2$.
\end{remark}

We can now start the proof of the main result that indecomposable elements from quadratic fields often remain indecomposable in biquadratic extensions. In the rest of the section, the proofs and sometimes also the statements of results will need to distinguish the cases 1.-4. for the triple $p, q, r$.

\begin{proposition} \label{prop:case q}
In the cases 1., 2. and 3., the indecomposable integers in $\Q(\sqrt{q})$ are indecomposable in $\Q(\sqrt{p},\sqrt{q})$. 
\end{proposition}

\begin{proof}
Let $x+y\omega_q$ be an indecomposable integer in $\Q(\sqrt{q})$ where 
\[
    \omega_q=\left\{
                \begin{array}{ll}
                  \sqrt{q}\quad \text{ in case 1,}\\
                  \frac{1+\sqrt{q}}{2}\text{ in cases 2 and 3.}
                \end{array}
              \right.
  \]
Assuming that $x+y\omega_q$ is decomposable in $\Q(\sqrt{p},\sqrt{q})$, we can express our element as a sum of two totally positive integers, thus
\[
x+y\omega_q=\underbrace{a+b\sqrt{p}+c\omega_q+d\frac{\sqrt{p}+\sqrt{r}}{2}}_{\alpha}+\underbrace{a'-b\sqrt{p}+c'\omega_q-d\frac{\sqrt{p}+\sqrt{r}}{2}}_{\alpha'}.
\]    
Since the elements $\alpha$ and $\alpha'$ are totally positive, from Lemma \ref{lemma:totpos} we conclude that $a+c\omega_q>0$, $a+c\overline{\omega_q}>0$, $a'+c'\omega_q>0$ and $a'+c'\overline{\omega_q}>0$. From this it follows that
\[
x+y\omega_q=(a+c\omega_q)+(a'+c'\omega_q)
\]
where $a+c\omega_q$ and $a'+c'\omega_q$ are totally positive. This contradicts our assumption that $x+y\omega_q$ is indecomposable in $\Q(\sqrt{q})$.   
\end{proof}

Note that in the previous proposition, it is not possible to interchange the role of $p$ and $q$.

\begin{proposition} \label{th:conve}
If $\sqrt{q}>\gcd(p,q)$, then the totally positive convergents of $-\overline{\omega_p}$ are indecomposable in $\Q(\sqrt{p},\sqrt{q})$.
\end{proposition}
\begin{proof}
We first prove the proposition in the cases 1., 2. and 3. Let $x+y\sqrt{p}$ be a convergent of $\sqrt{p}$ and suppose, contrary to our claim, that it is decomposable in $\Q(\sqrt{p},\sqrt{q})$, so
\[
x+y\sqrt{p}=\underbrace{\frac{a}{2}+\frac{b}{2}\sqrt{p}+\frac{c}{2}\sqrt{q}+\frac{d}{2}\sqrt{r}}_{\alpha}+\underbrace{\frac{a'}{2}+\frac{b'}{2}\sqrt{p}-\frac{c}{2}\sqrt{q}-\frac{d}{2}\sqrt{r}}_{\alpha'}.
\]
Integers $a,a',c$ and $c'$ are all even in the case 1. Without loss of generality, we can assume that $y>0$. We next claim that either $c$ is equal to zero or $c$ and $d$ are both positive or both negative, which follows from the facts that $\frac{a}{2}-\frac{b}{2}\sqrt{p}<1$ and $\frac{|c|}{2}\sqrt{q}>1$ for nonzero $c$, even if $q\equiv 1\;(\text{mod } 4)$ where $q\geq 5$. Futhermore, $d\neq 0$. If $d$ were equal to zero, it would contradict the indecomposability of $x+y\sqrt{p}$ in $\Q(\sqrt{p})$ for $c=0$ or it would contradict the total positivity of $\alpha$ for $c\neq 0$.  

Let $k=\gcd(p,q)$. By our assumption, $\sqrt{q}>k$, which gives $p<r$. From the total positivity of the elements $\alpha$ and $\alpha'$ we conclude that both $\frac{a}{2}$ and $\frac{a'}{2}$ are positive. Therefore, one of these numbers is less than or equal to $\frac{x}{2}$. 

Suppose that $|d|>y$. Then
\[
\frac{|d|}{2}\sqrt{r}>\frac{y+1}{2}\sqrt{p}>\frac{x}{2},
\] 
the last inequality being a consequence of the fact that $x+y\sqrt{p}$ is a convergent of $\sqrt{p}$, so $x-y\sqrt{p}<1$ and $x-(y+1)\sqrt{p}<0$. This contradicts our assumption that both $\alpha$ and $\alpha'$ are totally positive. Therefore, it must hold that $|d|\leq y$. 

Moreover, we conclude that one of $\frac{a}{2}-\frac{b}{2}\sqrt{p}$ and $\frac{a'}{2}-\frac{b'}{2}\sqrt{p}$ is less than or equal to $\frac{x}{2}-\frac{y}{2}\sqrt{p}$, which follows from the total positivity of our  expressions. From Lemma \ref{lemma:totpos} we know that $\frac{a}{2}-\frac{b}{2}\sqrt{p}>0$ and $\frac{a'}{2}-\frac{b'}{2}\sqrt{p}>0$. If both of these expressions were greater than $\frac{x}{2}-\frac{y}{2}\sqrt{p}$, then their sum would be greater than $x-y\sqrt{p}$, which is impossible.    

Suppose without loss of generality that this condition is satisfied by $\alpha$, i.e., $\frac{a}{2}-\frac{b}{2}\sqrt{p}\leq \frac{x}{2}-\frac{y}{2}\sqrt{p}$.   One of $\frac{c}{2}\sqrt{q}-\frac{d}{2}\sqrt{r}$ and $-\frac{c}{2}\sqrt{q}+\frac{d}{2}\sqrt{r}$ is negative. For $\alpha$ to be totally positive, necessarily
\[
\frac{a}{2}-\frac{b}{2}\sqrt{p}>\left|\frac{c}{2}\sqrt{q}-\frac{d}{2}\sqrt{r}\right|,
\]
where we can suppose that $c$ and $d$ are nonnegative.

The convergent $x+y\sqrt{p}$ is also a best approximation of the second kind, thus
\[
x-y\sqrt{p}\leq|kc-d\sqrt{p}|
\]
for all $d\leq y$, where equality can occur only when $kc=x$ and $d=y$. From this we obtain
\[
\frac{x}{2}-\frac{y}{2}\sqrt{p}\leq\left|\frac{kc}{2}-\frac{d}{2}\sqrt{p}\right|<\frac{\sqrt{q}}{k}\left|\frac{kc}{2}-\frac{d}{2}\sqrt{p}\right|=\left|\frac{c}{2}\sqrt{q}-\frac{d}{2}\sqrt{r}\right|
\]
for all $c,d\in\N_{0}$ satisfying $d\leq y$. Hence it is not possible to find $c$ and $d$ such that $\alpha$ is totally positive, which gives us a contradiction. Therefore, the element $x+y\sqrt{p}$ is indecomposable in $\Q(\sqrt{p},\sqrt{q})$. 

The proof of case 4 is similar to the previous part, but there are some points which should better be outlined. In this case, we have convergents of the form $x+y\frac{1+\sqrt{p}}{2}$, where $y$ can be a positive integer. In both cases 4a and 4b, we can express our possible decomposition as
\[
x+y\frac{1+\sqrt{p}}{2}=\underbrace{\frac{a}{2}+\frac{b}{2}\cdot\frac{1+\sqrt{p}}{2}+\frac{c}{2}\sqrt{q}+\frac{d}{4}(\sqrt{q}+\sqrt{r})}_{\alpha}+\underbrace{\frac{a'}{2}+\frac{b'}{2}\cdot\frac{1+\sqrt{p}}{2}-\frac{c}{2}\sqrt{q}-\frac{d}{4}(\sqrt{q}+\sqrt{r})}_{\alpha'}.
\]
Similarly to the previous part, we can see that $0<|d|\leq y$, $\frac{c}{2}+\frac{d}{4}$ and $\frac{d}{4}$ are both positive or both negative or $\frac{c}{2}+\frac{d}{4}=0$. Since $x+y\frac{1+\sqrt{p}}{2}$ is a convergent of $\frac{\sqrt{p}-1}{2}$, we have $x+y\frac{1-\sqrt{p}}{2}<1$, thus one of $\frac{a}{2}+\frac{b}{2}\cdot\frac{1-\sqrt{p}}{2}$ or $\frac{a'}{2}+\frac{b'}{2}\cdot\frac{1-\sqrt{p}}{2}$ is less than or equal to $\frac{1}{2}$. If $\frac{c}{2}+\frac{d}{4}=0$, then necessarily $d\neq 0$, othewise our convergent is decomposable in $\Q(\sqrt{p})$. If $\frac{c}{2}+\frac{d}{4}\neq0$, then $\left|\frac{c}{2}+\frac{d}{4}\right|\sqrt{q}>\frac{1}{2}$ and the expressions $\frac{c}{2}+\frac{d}{4}$ and $\frac{d}{4}$ have the same sign.       

Since $p$ and $q$ are odd, $k$ is also odd and $(k-1)d$ is even. Therefore, the expression $kc+\frac{kd}{2}$ can be rewritten as $\left(kc+\frac{(k-1)d}{2}\right)+\frac{d}{2}$, where $kc+\frac{(k-1)d}{2}$ is an integer. Hence we require fulfillment of the condition
\[
\frac{x}{2}-\frac{y}{2}\frac{\sqrt{p}-1}{2}\geq\frac{a}{2}+\frac{b}{2}\cdot\frac{1-\sqrt{p}}{2}>\frac{\sqrt{q}}{2k}\left|\left(kc+\frac{(k-1)d}{2}\right)-d\frac{\sqrt{p}-1}{2}\right|.
\]
If $kc+\frac{(k-1)d}{2}$ is negative, we get the estimate
\[
\left|\left(kc+\frac{(k-1)d}{2}\right)-d\frac{\sqrt{p}-1}{2}\right|>\left|0-d\frac{\sqrt{p}-1}{2}\right|.
\] 
The rest is similar to the previous part.
\end{proof}

In the next proposition, we require a stronger condition for $\gcd(p,q),$ but it holds also for semiconvergents.

\begin{proposition} \label{th:semi}
If
\[
\sqrt{q}>\gcd(p,q)\max\{u_i;\; i\text{ odd, }[u_0,\overline{u_1,u_2,\dots, u_{s-1}, u_s}] \text{ is the continued fraction of } -\overline{\omega_p} \},
\]
then the indecomposable integers from $\Q(\sqrt{p})$ are indecomposable in $\Q(\sqrt{p},\sqrt{q})$.
\end{proposition}

\begin{proof}
We begin by proving the cases 1., 2. and 3. In Proposition \ref{th:conve}, we showed the indecomposability of positive convergents of $\sqrt{p}$. It remains to prove this property for other semiconvergents of $\sqrt{p}$.  

To obtain a contradiction, assume that the semiconvergent $x+y\sqrt{p}$ is decomposable in $\Q(\sqrt{p},\sqrt{q})$. Then
\[
x+y\sqrt{p}=\underbrace{\frac{a}{2}+\frac{b}{2}\sqrt{p}+\frac{c}{2}\sqrt{q}+\frac{d}{2}\sqrt{r}}_{\alpha}+\underbrace{\frac{a'}{2}+\frac{b'}{2}\sqrt{p}-\frac{c}{2}\sqrt{q}-\frac{d}{2}\sqrt{r}}_{\alpha'},
\]
as in the previous proof. 
 Furthermore, we can assume that 
\[
\frac{x}{2}-\frac{y}{2}\sqrt{p}\geq \frac{a}{2}-\frac{b}{2}\sqrt{p}
\] 
and it must again hold that $0<|d|\leq y$, the coefficients $c$ and $d$ have the same sign or $c=0$, and  
\[
\frac{a}{2}-\frac{b}{2}\sqrt{p}>\left|\frac{c}{2}\sqrt{q}-\frac{d}{2}\sqrt{r}\right|.
\]

The semiconvergent $x+y\sqrt{p}$ is not a best approximation of the second kind, so there exist $u,v\in\N_{0}$ such that
\[
|u-v\sqrt{p}|<x-y\sqrt{p},
\]
where $1\leq v\leq y$ and $\frac{u}{v}\neq\frac{x}{y}$. We can rewrite $x+y\sqrt{p}$ as
\[
x+y\sqrt{p}=x_n+y_n\sqrt{p}+l(x_{n+1}+y_{n+1}\sqrt{p}),
\]
where $l\in\N$, $x_n+y_n\sqrt{p}$ and $x_{n+1}+y_{n+1}\sqrt{p}$ are consecutive convergents of $\sqrt{p}$ and $n$ is odd.

We claim that 
\[
|x_{n+1}-y_{n+1}\sqrt{p}|\leq |u-v\sqrt{p}|
\]
for every $u,v\in\N_{0}$ such that $1\leq v\leq y$ and $|u-v\sqrt{p}|\leq x-y\sqrt{p}$. It follows from the facts that $x_{n+1}+y_{n+1}\sqrt{p}$ is a best approximation of the second kind and the next lowest value is $|x_{n+2}-y_{n+2}\sqrt{p}|$, for which $y_{n+2}>y$.

Let $t>0$. The inequality
\[
t(y_{n+1}\sqrt{p}-x_{n+1})\leq x_n-y_n\sqrt{p}+l(x_{n+1}-y_{n+1}\sqrt{p})=x-y\sqrt{p}
\] 
is equivalent to
\[
0\leq\underbrace{x_n-y_n\sqrt{p}}_{>0}+(l+t)(\underbrace{x_{n+1}-y_{n+1}\sqrt{p}}_{<0}).
\]
We will show that the inequality cannot hold if $l+t>u_{n+2}+1$, where $[u_0,\overline{u_1,u_2,\dots, u_{s-1}, u_s}]$ is the continued fraction of $\sqrt{p}$. To prove this, we rewrite $l+t$ as $u_{n+2}+z$ where $z>1$. Then
\[
x_n-y_n\sqrt{p}+(u_{n+2}+z)(x_{n+1}-y_{n+1}\sqrt{p})=\underbrace{x_{n+2}-y_{n+2}\sqrt{p}}_{>0}+z(\underbrace{x_{n+1}-y_{n+1}\sqrt{p}}_{<0})<0,
\]
which follows from the fact that $x_{n+2}-y_{n+2}\sqrt{p}<|x_{n+1}-y_{n+1}\sqrt{p}|$. 

Let us take $\frac{\sqrt{q}}k$ for $t$, $kc$ for $u$ and $d$ for $v$ above. For any pair $kc$ and $d$, one of the following options occurs. Either $x-y\sqrt{p}< |kc-d\sqrt{p}|$, which leads to the similar situation as in \ref{th:conve}, or $|kc-d\sqrt{p}|\leq x-y\sqrt{p}$, which we discussed in the previous part of the proof.  

Then, since we assume that  $\frac{\sqrt{q}}{k}>u_{n+2}$, we have
\[
\frac{a}{2}-\frac{b}{2}\sqrt{p}\leq\frac{x}{2}-\frac{y}{2}\sqrt{p}<\frac{\sqrt{q}}{2k} (y_{n+1}\sqrt{p}-x_{n+1})\leq\frac{\sqrt{q}}{k}\left|\frac{kc}{2}-\frac{d}{2}\sqrt{p}\right|=\left|\frac{c}{2}\sqrt{q}-\frac{d}{2}\sqrt{r}\right|
\] 
for all $d\leq y$, which is a contradiction. 

Therefore, if 
\[
\sqrt{q}>k\max\{u_i;\; i\text{ odd, }[u_0,\overline{u_1,u_2,\dots, u_{s-1}, u_s}] \text{ is the continued fraction of } \sqrt{p} \},
\]
then all totally positive semiconvergents of $\sqrt{p}$ are indecomposable in $\Q(\sqrt{p},\sqrt{q})$.

The proof of case 4 is similar.
\end{proof}

Note that in case 4 of the previous theorem, $p$ and $q$ are interchangeable.

Summarizing our propositions and collecting the conclusions about the three quadratic subfields, we get the proof of Theorem \ref{Indecomposables}.

\begin{proof}[Proof of Theorem \ref{Indecomposables}]
First note that the condition $\sqrt{r}>\sqrt{p}$ is equivalent to the condition $\sqrt{q}>\gcd(p,q)$, and similarly for the other cases. Part (a) follows directly from Proposition \ref{th:conve} and Proposition \ref{th:semi}. Part (b) follows from the same propositions using the fact that in the case 4., $p$ and $q$ are interchangeable, and from Proposition \ref{prop:case q}. In part (c) we are using the fact that $p$ and $r$ are interchangeable.
\end{proof}

%-----------------------------------------------------
\section{Estimates for algorithmic computations}

The estimates from this section allow us to convert infinite problems to finite ones, which can be solved algorithmically.

The following lemma can be useful when investigating if an element from a quadratic field that is not a square can become a square in a biquadratic extension.

\begin{lemma} \label{squares} 
Let $F=\Q(\sqrt{p})$ and $\alpha=a+b\sqrt{p}\in\OK F$ with $a,b\neq 0$. Suppose that $\alpha$ is not a square in $\OK F$ but becomes a square in $K=\BQ p q$. Then if $\beta=x+y\sqrt{p}+z\sqrt{q}+w\sqrt{r}\in\OK{K}$ is such that $\alpha=\beta^2,$ it must hold that $x=y=0.$
\end{lemma}

\begin{proof}
	We have $$\begin{aligned}a+b\sqrt{p}&=(x+y\sqrt{p}+z\sqrt{q}+w\sqrt{r})^2=\\&=x^2+py^2+qz^2+rw^2+2y\sqrt{p}(z\sqrt{q}+w\sqrt{r}+x)+2z\sqrt{q}(w\sqrt{r}+x)+2xw\sqrt{r},\end{aligned}$$ where all the coefficients $a,b,x,y,z,w\in\frac 14\Z$.
    Let $k=\gcd(p,q)$. Using the fact that $r=\frac{pq}{k^2}$ and comparing the coefficients of different square roots in the equation above, we get the following system of equations:
\begin{eqnarray} 
    a & = &x^2+py^2+qz^2+rw^2 \label{1}\ ,\\
    b & = &2xy+2zw\frac q k \label{2}\ ,\\
    0 & = &xz+yw\frac p k \label{3}\ ,\\
    0 & = &xw+kyz\label{4}.
\end{eqnarray}

First suppose that $p\nmid q,$ so that there exists a prime $p_1$ such that $p_1\mid p$ and $p_1\nmid q$ (so also $p_1\nmid k$). For a rational number $s$, let $v_{p_1}(s)$ denote its $p_1$-adic valuation. From equations \eqref{3} and \eqref{4}, we get respectively \begin{eqnarray*}v_{p_1}(x)+v_{p_1}(z) &=& v_{p_1}(y) + v_{p_1}(w) + 1 \ ,\\ v_{p_1}(x) + v_{p_1}(w) &=& v_{p_1}(y) + v_{p_1}(z). \end{eqnarray*} If $xyzw\neq0,$ all quantities above are finite, so we can add the equations together. This yields $2v_{p_1}(x)=2v_{p_1}(y)+1,$ which is not possible. Hence it must hold that $xyzw=0$.

If $w=0$, since we assumed that $b\neq0$, $\eqref{2}$ gives that $xy\neq0$. Then from $\eqref{3}$ we see that $z=0$, so we get $a+b\sqrt{p}=(x+y\sqrt{p})^2$. If we further note that $x,y$ can be half-integers only if $a$ or $b$ is a half-integer, we can conclude that $x+y\sqrt{p}\in \OK F$, which contradicts the assumption that $\alpha$ is not a square.

If $z=0$, we again get from $\eqref{2}$ that $xy\neq0$ and from $\eqref{3}$ that $w=0,$ leading to the same contradiction as above.

The last possibility is when $xy=0$. Then $\eqref{2}$ gives us that $zw\neq 0$, so from $\eqref{3}$ we see that both $x$ and $y$ are $0$, which we wanted to prove.

It remains to solve the case when $p|q$. Then $k=p$ and we can rewrite equations $\eqref{3}$ and $\eqref{4}$ as follows: \begin{eqnarray*}0 &=&xz+yw,\\ 0 &=& xw+pyz.\end{eqnarray*} We can proceed in a similar way as above, for $p_1$ taking an arbitrary prime divisor of $p$.
\end{proof}

It is worth noting that Lemma \ref{squares} implies that the square root of $x+y\sqrt{p}$ in the biquadratic field $K$, if exists, is never totally positive. Hence, if an indecomposable element from a quadratic field becomes square in a biquadratic field, then its square root cannot be indecomposable.

To give some estimates on rational coefficients of an element of $\OK{K}$ (which satisfies some initial condition), we often take advantage of the trace function, as in the following lemma.

\begin{lemma}\label{Lemma:TraceSquare}
Let $\alpha, \beta \in \OK{K}$ and $\beta=a+b\sqrt{p}+c\sqrt{q}+d\sqrt{r}$ for some $a, b, c, d\in\Q$.

\begin{enumerate}[(a)]
\item If $0\prec\beta\prec\alpha$, then
\begin{equation}\label{Eq:estimates1}
|a|\leq \frac14\Tr{\alpha}, \  \ \ 
|b|\leq \frac1{4\sqrt p}\Tr{\alpha}, \ \ \ 
|c|\leq \frac1{4\sqrt q}\Tr{\alpha}, \ \ \ 
|d|\leq \frac1{4\sqrt r}\Tr{\alpha}.
\end{equation}

\item If $\beta^2\prec\alpha$, then
\begin{equation}\label{Eq:estimates}
a^2\leq \frac14\Tr{\alpha}, \  \ \ 
b^2\leq \frac1{4p}\Tr{\alpha}, \ \ \ 
c^2\leq \frac1{4q}\Tr{\alpha}, \ \ \ 
d^2\leq \frac1{4r}\Tr{\alpha}.
\end{equation}
\end{enumerate}
\end{lemma}

\begin{proof}
Part (a) follows from Lemma \ref{lemma:totpos} and from the fact that $\beta\prec\alpha$ implies $\Tr{\beta}<\Tr{\alpha}$. To prove part (b), note that $\Tr{\beta^2}<\Tr{\alpha}$ and $\Tr{\beta^2}=4\left(a^2+b^2p+c^2q+d^2r\right)$; therefore,
$$a^2,\; b^2p,\; c^2q,\; d^2r \leq \frac14\Tr{\beta^2} < \frac14\Tr{\alpha}. $$
\end{proof}

Part (a) of this lemma can be  directly applied to check whether a given $\alpha\in\OKPlus{K}$ is indecomposable.

\begin{corollary}\label{Cor:indecomposables}
Let $\alpha\in\OKPlus{K}$, and suppose that there do not exist $a,b,c,d\in\Q$ satisfying the inequalities in (\ref{Eq:estimates1}), and such that $a+b\sqrt{p}+c\sqrt{q}+d\sqrt{r}\in\OK{K}\setminus\{0\}$. Then $\alpha$ is indecomposable.
\end{corollary}

In the escalation method (that will be outlined in Section \ref{Section:Examples}), we are interested in finding all possible non-diagonal coefficients such that the resulting quadratic form is totally positive definite. The next proposition gives a necessary condition for such coefficients.

\begin{proposition} \label{Prop:nondiagonalcoef}
Let 
$$Q(\mathbf{x})=\sum_{i,j=1}^{n}{\alpha_{ij}x_ix_j}$$ 
be a classical totally positive definite quadratic form with coefficients from $\OK{K}$. Then 
$$\alpha_{ij}^2\prec\alpha_{ii}\alpha_{jj}$$
for every $i\neq j$. If  $\alpha_{ij}=a_{ij}+b_{ij}\sqrt{p}+c_{ij}\sqrt{q}+d_{ij}\sqrt{r}$ for some $a_{ij}, b_{ij}, c_{ij}, d_{ij}\in\Q$, then 
$$
a_{ij}^2\leq \frac14\Tr{\alpha_{ii}\alpha_{jj}}, \  \ \ 
b_{ij}^2\leq \frac1{4p}\Tr{\alpha_{ii}\alpha_{jj}}, \ \ \ 
c_{ij}^2\leq \frac1{4q}\Tr{\alpha_{ii}\alpha_{jj}}, \ \ \ 
d_{ij}^2\leq \frac1{4r}\Tr{\alpha_{ii}\alpha_{jj}}.
$$
\end{proposition}
\begin{proof}
The matrix $\left(\alpha_{ij}\right)_{i,j=1}^n$ is totally positive definite, and hence the $2\times 2$ minor
$$
\begin{vmatrix}
\alpha_{ii} & \alpha_{ij}\\
\alpha_{ij} & \alpha_{jj}
\end{vmatrix}
$$
has to be totally positive for every $i\neq j$; therefore, $\alpha_{ij}^2\prec\alpha_{ii}\alpha_{jj}$. 
The rest follows from part (b) of Lemma \ref{Lemma:TraceSquare}.
\end{proof}

Another question arising in the escalation method is to find elements which are not yet represented by the quadratic form; the following lemma describes some necessary conditions for an element of $\OKPlus{K}$ to be represented.

\begin{proposition}\label{Prop:BoundsForRepresentability}
Let $\gamma\in\OKPlus{K}$, and let $Q$ be a classical totally positive definite quadratic form with $n$ variables and coefficients from $\OK{K}$ representing $\gamma$. Let $\eta_i\in\OK{K}$, $i=1,\dots, n$, be such that $Q(\eta_1, \dots, \eta_n)=\gamma$, and let $\eta_i=a_i+b_i\sqrt{p}+c_i\sqrt{q}+d_i\sqrt{r}$ for some $a_i, b_i, c_i, d_i \in \Q$, $i=1, \dots, n$.
\begin{enumerate}[(a)]
	\item If $Q$ is diagonal, $Q(\textbf{x})=\sum_{i=1}^n{\alpha_ix_i^2}$, then 
    		$$\eta_i^2\preceq\frac{\gamma}{\alpha_i}.$$
            In particular, 
            $$ a_i^2\leq\frac14\Tr{\frac{\gamma}{\alpha_i}}, \ \ 
            	b_i^2\leq\frac{1}{4p}\Tr{\frac{\gamma}{\alpha_i}}, \ \ 
                c_i^2\leq\frac{1}{4q}\Tr{\frac{\gamma}{\alpha_i}}, \ \ 
                d_i^2\leq\frac{1}{4r}\Tr{\frac{\gamma}{\alpha_i}}.$$
    \item In the general case, we have  
    		$$\Tr{\eta_i^2} \leq \sum_{k=1}^4{\sum_{j=1}^n{\frac{\sigma_k(\gamma)}{\lambda_j^{(k)}}}},$$ where $\lambda_1^{(k)}, \dots, \lambda_n^{(k)}$ are the eigenvalues of the matrix $\sigma_k(Q)$. In particular,
            $$ a_i^2 \leq \frac{1}{4}\sum_{k=1}^4{\sum_{j=1}^n{\frac{\sigma_k(\gamma)}{\lambda_j^{(k)}}}}, \ \ 
b_i^2 \leq \frac{1}{4p}\sum_{k=1}^4{\sum_{j=1}^n{\frac{\sigma_k(\gamma)}{\lambda_j^{(k)}}}}, \ \ 
c_i^2 \leq \frac{1}{4q}\sum_{k=1}^4{\sum_{j=1}^n{\frac{\sigma_k(\gamma)}{\lambda_j^{(k)}}}}, \ \ 
d_i^2 \leq \frac{1}{4r}\sum_{k=1}^4{\sum_{j=1}^n{\frac{\sigma_k(\gamma)}{\lambda_j^{(k)}}}}$$
\end{enumerate}
\end{proposition}

\begin{proof}
(a) It follows from the totally positive definiteness of $Q$ that $\alpha_i\succ0$ for every $i=1,\dots, n$. Thus, 
$$\alpha_i\eta_i^2\preceq \sum_{j=1}^n{\alpha_j\eta_j^2} = \gamma,$$
and hence
$$\eta_i^2\preceq\frac{\gamma}{\alpha_i}.$$
The rest follows from part (b) of Lemma \ref{Lemma:TraceSquare}.

To prove (b), let $Q(\mathbf{x})=\sum_{i,j=1}^n{\alpha_{ij}x_ix_j}$, and for $k=1, \dots, 4$ denote $A^{(k)}=\left(\sigma_k(\alpha_{ij})\right)_{i,j=1}^n$. Clearly, the matrices $A^{(k)}$ are symmetric and positive definite, and hence orthogonal diagonalizable. Thus, for every $k=1, \dots, 4$, there exists a matrix $U^{(k)}$ such that $\left(U^{(k)}\right)^tA^{(k)}U^{(k)}$ is diagonal, and $\left(U^{(k)}\right)^tU^{(k)}=I_n$. Then $\left(U^{(k)}\right)^tA^{(k)}U^{(k)}=diag\left(\lambda_1^{(k)}, \dots, \lambda_n^{(k)} \right)$, where $\lambda_1^{(k)}, \dots, \lambda_n^{(k)}$ are the eigenvalues of the matrix $A^{(k)}$. Note that $\lambda_i^{(k)}>0$ for every $i=1, \dots n$ and $k=1, \dots, 4$. Let $\eta^{(k)}=(\sigma_k(\eta_1), \dots, \sigma_k(\eta_n))^t$, and $\rho^{(k)}=\left(\rho_1^{(k)}, \dots,\rho_n^{(k)}\right)=\left(U^{(k)}\right)^t\eta^{(k)}$; then \begin{multline*}
\sigma_k(\gamma)
=\left(\eta^{(k)}\right)^tA\eta^{(k)}
=\left(\eta^{(k)}\right)^tU^{(k)}diag\left(\lambda_1^{(k)}, \dots, \lambda_n^{(k)}\right)\left(U^{(k)}\right)^t\eta^{(k)} \\
=\left(\rho^{(k)}\right)^tdiag\left(\lambda_1^{(k)}, \dots, \lambda_n^{(k)}\right)^t\rho^{(k)}
=\sum_{i=1}^n{\lambda_i^{(k)}\left(\rho_i^{(k)}\right)^2}.
\end{multline*}
Therefore,
$$\left(\rho_i^{(k)}\right)^2\leq \frac{\sigma_k(\gamma)}{\lambda_i^{(k)}}$$
for every $i=1, \dots, n$ and $k=1,\dots,4$. Moreover, $\left\|\eta^{(k)}\right\|=\left\|\left(U^{(k)}\right)^t\eta^{(k)}\right\|$ by the orthogonality of $U^{(k)}$; thus,
\begin{multline*}
\sigma_k(\eta_i^2)
=\sigma_k(\eta_i)^2
\leq\sum_{j=1}^n{\sigma_k(\eta_j)^2}
=\left\|\eta^{(k)}\right\|^2
=\left\|\left(U^{(k)}\right)^t\eta^{(k)}\right\|^2
=\left\|\rho^{(k)}\right\|^2
=\sum_{j=1}^n{\left(\rho_j^{(k)}\right)^2}
\leq \sum_{j=1}^n\frac{\sigma_k(\gamma)}{\lambda_j^{(k)}}.
\end{multline*}
Summing over all $k$'s, we get
$$\Tr{\eta_i^2}=\sum_{k=1}^4{\sigma_k(\eta_i^2)}
\leq \sum_{k=1}^4{\sum_{j=1}^n{\frac{\sigma_k(\gamma)}{\lambda_j^{(k)}}}};$$
the rest follows from the fact that $a_i^2+b_i^2p+c_i^2q+d_i^2r=\frac14\Tr{\eta_i^2}$.
\end{proof}

%-------------------------------------------------------
\section{Examples}\label{Section:Examples}

In this final section, we will construct certain (classical, totally positive definite) quadratic forms over some particular biquadratic number fields. For that, we will use the escalation method (see \cite{BH}): starting with the quadratic form $Q_1(\mathbf{x})=x_1^2$ and proceeding in steps, we will add as coefficients of the new variables exactly the elements, which are not represented yet. As we would like to obtain a lower bound on the number of variables of a universal quadratic form, we need to be careful and consider all possible quadratic forms: when adding a new variable, new nondiagonal terms arise. The only restriction for the coefficients of these nondiagonal terms is that the resulting quadratic form has to be totally positive definite (and classical); that is where Proposition \ref{Prop:nondiagonalcoef} comes into account. 

Using the escalation method, one can prove the following proposition, which is a version of \cite[Proposition 2]{Ka} for classical quadratic forms.

\begin{proposition}\label{nondiagonals}
Assume that there exist totally positive elements $\alpha_1, \alpha_2, \dots, \alpha_n\in\OK{K}$ such that for all $1\leq i\neq j\leq n$ the following holds: if $\alpha_i\alpha_j\succeq\gamma^2$ for $\gamma\in\OK{K}$, then $\gamma=0$. Then there are no universal totally positive $(n-1)$-ary classical quadratic forms over $\OK{K}$.
\end{proposition}

Note that the quadratic form resulting from this proposition would be diagonal. Setting $\alpha_1=1$, it becomes clear why indecomposable elements are important: if $\alpha_j$ is indecomposable (and it is not a square), then there is no nonzero $\gamma\in\OK{K}$ such that  $1\cdot\alpha_j\succeq\gamma^2$. Furthermore, as long as the quadratic form is diagonal, it is relatively easy to check if an indecomposable element is already represented or not.

\begin{lemma}
Let $Q(\mathbf{x})=\sum_{i=1}^n{\alpha_ix_i^2}$ be a totally positive definite quadratic form over $K$, and let $\beta\in\OK{K}$ be indecomposable. Then $\beta$ is represented by $Q$ if and only if there exists $i\in\{1,\dots, n\}$ such that $\frac{\beta}{\alpha_i}$ is a square in $\OK{K}$.
\end{lemma}

\begin{proof}
Let $\beta=Q(\gamma_1, \dots, \gamma_n)$ for some $\gamma_1, \dots, \gamma_n \in\OK{K}$. It follows from the indecomposability of $\beta$ that there exists $i\in\{1,\dots, n\}$ such that $\beta=\alpha_i\gamma_i^2$. On the other hand, if $\frac{\beta}{\alpha_i}=\gamma^2$ for an element $\gamma\in\OK{K}$, then $\beta=Q(0, \dots, 0, \gamma, 0, \dots, 0)$, where $\gamma$ is on the $i$-th position.
\end{proof}

%----------------------------------
\subsection{Escalation over $\BQ23$}
Before starting with the escalation itself, we need to find some indecomposable elements. Let us take a look at the biquadratic field $K=\BQ 2 3$. The ring of integers of this field is
$$\OK{K}=\left[1, \sqrt2, \sqrt3, \frac{\sqrt2+\sqrt6}{2} \right]_{\Z}$$ (as a $\Z$-module).
The fundamental units of the quadratic subfields are $\varepsilon_1 =1+\sqrt2$, $\varepsilon_2 =2+\sqrt3$, $\varepsilon_3 =5+2\sqrt6$. Note that $\varepsilon_2$ and $\varepsilon_3$ are squares in $\BQ23$, because $\varepsilon_2=\left(\frac12\sqrt2+\frac12\sqrt6\right)^2$, $\varepsilon_3=\left(\sqrt2+\sqrt3\right)^2$. Hence, the multiplicative group $\UK{K}$ of units of $\OK{K}$ is generated by $\varepsilon_1, \sqrt{\varepsilon_2}, \sqrt{\varepsilon_3}$, i.e., 
$$\UK{K}=\left\langle 1+\sqrt2, \frac{\sqrt2+\sqrt6}{2}, \sqrt2+\sqrt3 \right\rangle$$ (for reference, see \cite[pg.2]{unitsE}). Note that none of the generators is totally positive; one can see that the only totally positive units in $\BQ23$ are powers and conjugates of powers of 
$$\mu=\left(1+\sqrt2\right)\frac{\sqrt2+\sqrt6}{2}\left(\sqrt2+\sqrt3\right)= 4+\frac52\sqrt2+2\sqrt3+\frac32\sqrt3.$$
It follows from Corollary \ref{Prop:indecompunits} that $\mu$ (and all its powers) are indecomposable.

The indecomposable elements (up to the multiplication by a unit) of the quadratic subfields of $\BQ23$ are  the following ones:
\begin{itemize}
	\item $\Q(\sqrt2)$: $1$, $2+\sqrt2$, $3+2\sqrt2$
	\item $\Q(\sqrt3)$: $1$, $2+\sqrt3$
	\item $\Q(\sqrt6)$: $1$, $3+\sqrt6$, $5+2\sqrt6$
\end{itemize}
It follows from Theorem \ref{Indecomposables} that all indecomposables from the fields $\Q(\sqrt2)$ and $\Q(\sqrt3)$ remain indecomposable in $\BQ23$; for the indecomposables from $\Q(\sqrt6)$, it has to be checked algorithmically using Corollary \ref{Cor:indecomposables}. Because the first step of the method of escalation is to consider the quadratic form $x_1^2$, which represents all squares from $\OKPlus{K}$, we need to cross out all the elements from the list above which are squares in $\BQ23$. Since $3+2\sqrt2=\left(1+\sqrt2\right)^2$, $2+\sqrt3=\left(\frac12\sqrt2+\frac12\sqrt6\right)^2$, $5+2\sqrt6=\left(\sqrt2+\sqrt3\right)^2$, we are left only with two indecomposables: $2+\sqrt2$ and $3+\sqrt{6}$. Furthermore, it holds that 
$$ \frac{2+\sqrt2}{\mu}=\left(1+\frac12\sqrt2-\frac12\sqrt6\right)^2;$$
hence, any quadratic form representing $\mu$ also represents $2+\sqrt2$. On the other hand, one can check (by a computer using Corollary \ref{Cor:indecomposables}) that $\zeta=3-\frac12\sqrt2-\sqrt3+\frac12\sqrt6$ is an indecomposable not arising from any quadratic subfield.

Putting everything together, from the point of view of escalation, we are left with three interesting indecomposable elements:
\begin{align*}
\mu&=4+\frac52\sqrt2+2\sqrt3+\frac32\sqrt6, & \norm{\mu}&=1\\
\sigma&=3+\sqrt6, & \norm{\sigma}&=9\\
\zeta&=3-\frac12\sqrt2-\sqrt3+\frac12\sqrt6, & \norm{\zeta}&=25
\end{align*}
Some other interesting elements can be obtained by multiplying the previous ones:
\begin{align*}
\frac{\sigma}{\mu}&=3-\frac32\sqrt2-\sqrt3+\frac12\sqrt6, & \norm{\frac{\sigma}{\mu}}&=9, \\
\frac{\zeta}{\mu}&=5-\sqrt2-2\sqrt3, & \norm{\frac{\zeta}{\mu}}&=25.
\end{align*}
Note that multiplication by a totally positive unit does not affect indecomposability; therefore, both of these elements are indecomposable.

We are finally prepared to start with the escalation. We will add new variables with coefficients from the two lists above.
\begin{enumerate}
	\item $Q_1(\mathbf{x})=x_1^2$
    
    \item $Q_2(\mathbf{x})=x_1^2+\mu x_2^2+2\alpha x_1 x_2$; \\
    since $\mu$ is indecomposable, the only possibility for the form to be positive definite is when $\alpha=0$. Hence, we have the quadratic form $Q_2(\mathbf{x})=x_1^2+\mu x_2^2$. This form does not represent $\sigma$, because $\sigma$ is indecomposable, and one can check that neither $\sigma$ nor $\frac{\sigma}{\mu}$ are squares.
    
	\item $Q_3(\mathbf{x})=x_1^2+\mu x_2^2+\sigma x_3^2+2\alpha x_1x_3+2\beta x_2x_3$; \\
    the matrix is
    $$Q_3=\begin{pmatrix}
    1 & 0 & \alpha \\
    0 & \mu & \beta \\
    \alpha & \beta & \sigma \\
    \end{pmatrix}$$
   From the indecomposability of $\sigma$ and $\mu\sigma$ follows that the only solutions of $\alpha^2\prec \sigma$ and $\beta^2\prec\mu\sigma$ are $\alpha,\beta=0$. Thus, we have a diagonal quadratic form $Q_3(\mathbf{x})=x_1^2+\mu x_2^2+\sigma x_3^2$. As $\frac{\sigma}{\mu}$ is  indecomposable, and non of the elements $\frac{\sigma}{\mu}$, $\frac{\sigma}{\mu^2}$, $\frac{\sigma}{\mu\sigma}$ is a square, $\frac{\sigma}{\mu}$ is not represented by this quadratic form.
     
    \item $Q_4(\mathbf{x})=x_1^2+\mu x_2^2+\sigma x_3^2+\frac{\sigma}{\mu} x_4^2 +2\alpha x_1x_4+2\beta x_2x_4 + 2\gamma x_3x_4$;\\
    the matrix is
    $$Q_4=\begin{pmatrix}
				1 & 0 & 0 & \alpha \\
				0 & \mu & 0 & \beta \\
				0 & 0 & \sigma & \gamma \\
				\alpha & \beta & \gamma & \frac{\sigma}{\mu} \\
			\end{pmatrix}$$
    \begin{itemize}
        \item Recall that both $\sigma$ and $\frac{\sigma}{\mu}$ are indecomposable. Hence, from the positive definiteness of the submatrices
        	$$\begin{pmatrix}
        	1 & \alpha\\
        	\alpha & \frac{\sigma}{\mu} \\
        	\end{pmatrix}, \ \ \ 
            \begin{pmatrix}
        	\mu & \beta\\
        	\beta & \frac{\sigma}{\mu} \\
        	\end{pmatrix}$$
            follows that the only possibilities are $\alpha=0$ and $\beta=0$.
        
        \item By a computation, the only solution of $\gamma^2\prec\frac{\sigma^2}{\mu}$ is again $\gamma=0$.
     \end{itemize}
Hence, we have a diagonal quadratic form $Q_4(\mathbf{x})=x_1^2+\mu x_2^2+\sigma x_3^2+\frac{\sigma}{\mu} x_4^2$; this quadratic form does not represent the indecomposable element $\zeta$.
\item $Q_5(\mathbf{x})=x_1^2+\mu x_2^2+\sigma x_3^2+\frac{\sigma}{\mu} x_4^2+\zeta x_5^2+2\alpha x_1x_5 +2\beta x_2x_5 +2\gamma x_3x_5+2\delta x_4x_5$;\\
this quadratic form has the matrix
  $$Q_5=\begin{pmatrix}
				1 & 0 & 0 & 0 &  \alpha \\
				0 & \mu & 0 & 0 & \beta \\
				0 & 0 & \sigma & 0 & \gamma \\
				0 & 0 & 0&  \frac{\sigma}{\mu} & \delta \\
                \alpha & \beta & \gamma & \delta & \zeta\\
			\end{pmatrix}.$$
  \begin{itemize}
  	\item Indecomposability of $\zeta$ and $\mu\zeta$ implies that $\alpha,\beta=0$.
    \item A computation gives the following solutions of $\gamma^2\prec\sigma\zeta$: 
  $$\gamma\in\left\{0, \pm\left(-1+\frac12\sqrt2-\frac12\sqrt6\right)\right\}$$
  \item $\delta^2\prec\frac{\sigma\zeta}{\mu}$ has, by a computation, solutions
  \begin{multline*}\delta\in\left\{0, \pm\left( -2+\frac12\sqrt2+\sqrt3-\frac12\sqrt6\right), \pm\left(-2+\frac32\sqrt2+\sqrt3-\frac12\sqrt6 \right), \right. \\
  \left. \pm\left( -1+\frac12\sqrt2+\sqrt3-\frac12\sqrt6\right), \pm\left(-\frac12\sqrt2+\frac12\sqrt6 \right)  \right\} \end{multline*}
  \end{itemize}
Hence, we have the quadratic form $Q_5(\mathbf{x})=x_1^2+\mu x_2^2+\sigma x_3^2+\frac{\sigma}{\mu} x_4^2+\zeta x_5^2+2\gamma x_3x_5+2\delta x_4x_5$ with 3 possibilities for $\gamma$, and 9 possibilities for $\delta$.
\end{enumerate}

We would like to continue the escalation by adding the element $\frac{\zeta}{\mu}$, but one would have to check all of the 27 possible forms $Q_5$ that none of them represents $\frac{\zeta}{\mu}$. Theoretically, this could be done by a calculation based on 
part (b) of Proposition \ref{Prop:BoundsForRepresentability}, but in practice, there are too many possibilities to be checked; for example, if $\gamma=-\frac12\sqrt2+\frac12\sqrt6$ and $\delta=0$, then there are more than $2\cdot10^{10}$ possibilities. Instead of that, let us make the following conclusion:
\begin{proposition}
A classical universal totally positive quadratic form over $\OK{K}$, $K=\BQ23$, must have at least 5 variables.
\end{proposition}

One may ask if we could get a better result by adding the elements $\mu, \sigma, \frac{\sigma}{\mu}, \zeta, \frac{\zeta}{\mu}$ in a different order. Calculations based on Proposition \ref{Prop:nondiagonalcoef} show that the number of elements $\gamma\in\OK{K}$ satisfying  $\alpha\beta\succeq\gamma^2$   for $\alpha, \beta \in\ \left\{1, \mu, \sigma, \frac{\sigma}{\mu}, \zeta, \frac{\zeta}{\mu}\right\}$, $\alpha\neq\beta$, is as indicated in Table \ref{tab1}.

$$\begin{array}{c|ccccc}
 \#\gamma & 1 &\mu &\sigma & \frac{\sigma}{\mu} & \zeta  \\
 \hline
\mu & 1 &  &  &  & \\
\sigma & 1 & 1 &  &  &   \\
\frac{\sigma}{\mu} & 1 & 1 & 1 &  &  \\
\zeta & 1 & 1 & 3 & 9 &\\
\frac{\zeta}{\mu} & 1 & 1 & 9 & 3 & 5\\
\end{array}$$
\captionof{table}{Number of $\gamma\in\OK{K}$ such that  $\alpha\beta\succeq\gamma^2$.} \label{tab1}

The first 3 rows indicate that we could add $1, \mu, \sigma, \frac{\sigma}{\mu}$ in any order, and we would always end with the diagonal quadratic form $Q_4$. On the other hand, any pair of the elements $\sigma, \frac{\sigma}{\mu}, \zeta, \frac{\zeta}{\mu}$ other than $\sigma, \frac{\sigma}{\mu}$ necessarily leads to a nondiagonal quadratic form.

\subsection{Escalation over $\BQ{6}{19}$}
Let us look at the biquadratic number field $K=\BQ{6}{19}$. We could proceed similarly as in the case of the field $\BQ23$, but rather than that, we will only pick some of the indecomposable elements from the quadratic subfields and skip most of the computations. Denote
$$
\alpha_1 = 1, \ \ 
\alpha_2=5+2\sqrt6, \ \ 
\alpha_3=3+\sqrt{6}, \ \ 
\alpha_4=9+2\sqrt{19}, \ \ 
\alpha_5=11+\sqrt{114}, \ \ 
\alpha_6=5+\sqrt{19}.
$$
All of these elements are indecomposable in the appropriate quadratic field; note that the indecomposablity in $\BQ{6}{19}$ of all of them but $\alpha_5$ follows from Theorem \ref{Indecomposables}, and the indecomposability of $\alpha_5$ can be checked by computer using Corollary \ref{Cor:indecomposables}. Let us write the number of elements $\gamma\in\OKPlus{K}$ such that $\gamma^2\preceq\alpha_i\alpha_j$ for all pairs $i\neq j$ in a table (see Table \ref{tab2}).
$$\begin{array}{c|ccccc}
 \#\gamma & 1 & \alpha_2 &\alpha_3 & \alpha_4 & \alpha_5  \\
 \hline
\alpha_2 &1&  &  &  &  \\
\alpha_3 &1 & 1 &  &  &    \\
\alpha_4 & 1 & 1 & 1 &  &   \\
\alpha_5 & 1 & 1 & 1 & 1 &  \\ 
\alpha_6 & 1 & 1 & 1 & 3 & 1  \\
\end{array}$$
\captionof{table}{Number of $\gamma\in\OK{K}$ such that  $\alpha_i\alpha_j\succeq\gamma^2$.} \label{tab2}
The first four rows show that, if proceeding by escalation, we would get a diagonal quadratic form $Q_5(\mathbf{x})=x_1^2+\alpha_2 x_2^2+\alpha_3 x_3^2 +\alpha_4 x_4^2+\alpha_5 x_5^2$, which clearly does not represent the indecomposable element $\alpha_6$, as $\frac{\alpha_6}{\alpha_i}$ is not square in $\BQ{6}{19}$ for any $i=1, \dots, 5$. Hence, we get the following result:

\begin{proposition}
A classical universal totally positive quadratic form over $\OK{K}$, $K=\BQ{6}{19}$, must have at least 6 variables.
\end{proposition}
%------------------------------------------------------
\section*{Acknowledgment}
We would like to thank V\'{\i}t\v{e}zslav Kala for encouraging us to start a student seminar and for his advice.\\

%-----------------------------------------------------
%-----------------------------------------------------
%\bibliographystyle{alpha}
%\bibliography{sample}

\end{document}